\documentclass[12pt]{amsart}
\usepackage[dvipdfm]{graphicx}
\usepackage{epsfig}
\usepackage{graphics}
\usepackage{amscd}
\usepackage{color}
\numberwithin{equation}{section}
\theoremstyle{plain}
\newtheorem{thm}{Theorem}[section]

\newtheorem{prop}[thm]{Proposition}
\newtheorem{lemma}[thm]{Lemma}

\theoremstyle{definition}
\newtheorem{dfn}[thm]{Definition}

\theoremstyle{remark}
\newtheorem{rem}[thm]{Remark}

\newcommand{\N}{\mathbb N}

\newcommand{\vu}{\text{\rm vu}}
\newcommand{\sign}{\text{\rm sign}}

\begin{document}
\title{Virtual unknotting numbers of certain virtual torus knots}
\author{Masaharu Ishikawa}
\address{Mathematical Institute, Tohoku University, Sendai, 980-8578, Japan}
\email{ishikawa@m.tohoku.ac.jp}
\author{Hirokazu Yanagi}
\address{Mathematical Institute, Tohoku University, Sendai, 980-8578, Japan}
\email{sara.320.410@gmail.com}
%\footnote[0]{This work is supported by MEXT, Grant-in-Aid for Young Scientists (B) (No. 22740032).}
\keywords{unknotting number, virtual knot}
\subjclass[2010]{Primary 57M25; Secondary 57M27}
%\date{14 January, 2017}

\begin{abstract}
The virtual unknotting number of a virtual knot is the minimal number of crossing changes that makes the virtual knot 
to be the unknot, which is defined only for virtual knots virtually homotopic to the unknot.
We focus on the virtual knot obtained from the standard $(p,q)$-torus knot diagram 
by replacing all crossings on one overstrand
into virtual crossings and prove that its virtual unknotting number is 
equal to the unknotting number of the $(p,q)$-torus knot, i.e. it is $(p-1)(q-1)/2$.
%In other words, we prove that the generalized unknotting number of this virtual knot is $(0,(p-1)(q-1)/2)$.
\end{abstract}

\maketitle

\section{Introduction}

The notion of virtual knots was introduced by Kauffman in~\cite{kauffman} as a generalization of classical knots.
A virtual knot diagram is a diagram obtained from a knot diagram by changing some of crossings into, so-called, virtual crossings.
A virtual crossing is depicted by a small circle around the crossing point, for example see Figure~\ref{fig1} below.
Two virtual knot diagrams are said to be equivalent if they are related by 
isotopy of diagrams and generalized Reidemeister moves shown in Figure~\ref{figRM}.
An equivalence class, or a diagram in this equivalent class, is called a {\it virtual knot}.
We can also define a virtual knot as an embedding of a circle into a thickened surface.
If a virtual knot can be modified into the unknot by generalized Reidemeiter moves and crossing changes
then we say that the virtual knot is {\it virtually null-homotopic}.
Here a {\it crossing change} is an operation which replaces an overstand at a classical crossing into an understrand. 
Remark that there exist virtual knots which cannot be modified into the unknot by only 
crossing changes.

There are several unknotting operations for virtual knots. For example, 
an operation which replaces a classical crossing of a diagram of a virtual knot by a virtual crossing
is called a {\it virtualization}\footnote{The terminology ``virtualization'' is used in~\cite{kauffman2} for a different operation.}, and any virtual knot can be modified into the unknot by applying virtualizations
successively. A forbidden move, introduced by Goussarov, Polyak and Viro~\cite{gpv}, 
is also known as an unknotting operation~\cite{kanenobu, nelson}.
A forbidden detour move, which is essentially introduced in~\cite{kanenobu, nelson},
is also, see~\cite{yoshiike}.

In~\cite{chernov}, Byberi and Chernov introduced the notion of the virtual unknotting number for 
virtually null-homotopic virtual knots~\cite{chernov}.
The {\it virtual unknotting number} $\vu(K)$ of a virtually null-homotopic virtual knot $K$
is the minimal number of crossing changes needed to modify $K$ into the unknot.
They studied virtual unknotting numbers for virtual knots with virtual bridge number one.

Recently, in~\cite{kkkp}, Kaur, Kamada, Kawauchi and Prabhakar introduced 
the notion of the generalized unknotting number of a virtual knot as follows:
For any virtual knot diagram, there exists a sequence of virtualizations and crossing changes such that
the virtual knot is modified into a virtual knot diagram virtually homotopic to the unknot.
For such a sequence $s$ for a virtual knot $K$, let $v_s$ and $c_s$ denote
the number of virtualizations and crossing changes in the sequence, respectively.
Then the generalized unknotting number $U(K)$ of a virtual knot $K$ is defined
by the minimal pair $(v_s,c_s)$ among such sequences and choices of diagrams of $K$, where
the minimality is defined by the lexicographic order.
For example, $\min\{(2,0),(0,1)\}=(0,1)$.
In~\cite{kkkp}, they determined the generalized unknotting numbers of virtual knots obtained from 
the standard $(2,p)$-torus knot diagrams by changing odd number of classical crossings into virtual crossings,
and also those of virtual knots obtained from the standard twisted knot diagrams by
changing some specific crossings into virtual crossings. 

In this paper, we study the virtual unknotting numbers for a certain class of
virtual knots obtained from the standard torus knot diagrams by applying virtualizations.
Let $D(T_{p,q})$ be the standard $(p,q)$-torus knot diagram, where $p$ is the braid index and $p,q>0$.
Here we do not care in which side we close the braid since we think of it as a diagram on $S^2$. 
The diagram has $q$ overstrands and we label them by $a_1,a_2,\ldots,a_q$ in a canonical order.
Fix an integer $n\in\{1,\ldots,q\}$ and virtualize all classical crossings
at which $a_i$ passes as an overstrand for $i=1,\ldots,n$. 
We denote the obtained virtual knot by $VT_{p,q}^n$.
Note that the supporting genus of $VT_{p,q}^n$ is at most one and this implies that
$VT_{p,q}^n$ is virtually null-homotopic, see Remark~\ref{rem99}.

\begin{figure}[htbp]
\begin{center}
  \includegraphics[width=6cm]{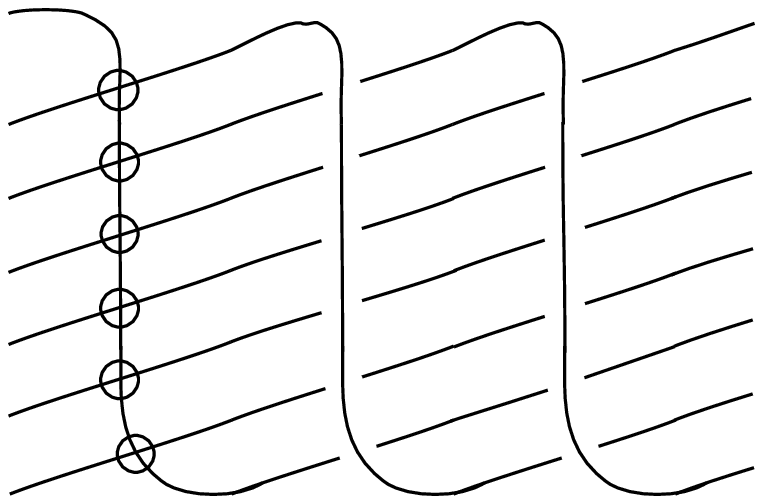}
  \caption{The closure of this braid is $VT_{7,3}^1$.\label{fig1}}
\end{center}
\end{figure}

The following is our main theorem.

\begin{thm}\label{thm1}
%The virtual unknotting number of $VT_{p,q}^1$ is
%\[
%   U(VT_{p,q}^1)=\left(0,\frac{(p-1)(q-1)}{2}\right).
%\]
%In other words, 
%$VT_{p,q}^1$ is virtually null-homotopic and 
The virtual unknotting number of the virtual knot $VT_{p,q}^1$ is 
\[
   \vu(VT_{p,q}^1)=\frac{(p-1)(q-1)}{2}.
\]
\end{thm}

We will prove this theorem by the same method as in~\cite{chernov}:
the lower bound is given by the $P$-invariant and the upper bound is done by demonstrating an unknotting sequence explicitly.

At the beginning of this study, we checked the lower bounds of the virtual unknotting numbers for
a few virtual knots obtained from the standard torus knot diagrams by applying virtualizations.
We may call such knots {\it virtual torus knots}.
It is well-known that the unknotting number of the $(p,q)$-torus knot is $(p-1)(q-1)/2$, which was conjectured by 
Milnor~\cite{milnor} and proved by Kronheimer and Mrowka~\cite{km}.
Examining a few examples, we found that the virtual unknotting number of $VT_{p,q}^1$ is exactly 
same as the unknotting number of the $(p,q)$-torus knot, which brought us the above theorem.
On the other hand, for the virtual knot $VT_{p,q}^2$, we are not sure if 
the lower bound obtained by the $P$-invariant attains the virtual unknotting number or not.
This will be discussed in Section~\ref{sec51}.
If we choose the virtualization on the standard torus knot diagram randomly, then 
the obtained virtual knot may not be virtually homotopic to the unknot, see Section~\ref{sec52}.

The authors would like to thank Andrei Vesnin for telling us the result of Kaur, Kamada, Kawauchi and Prabhakar.
They also thank Shin Sato for telling them the fact mentioned in Remark~\ref{rem99}
and thank Vladimir Chernov and Seiichi Kamada for helpful comments.
%answering a question about an invariant of A.~Henrich.
The first author is supported by the Grant-in-Aid for Scientific Research (C), JSPS KAKENHI Grant Number 16K05140.

\section{Preliminaries}

\subsection{Sign at a classical crossing}

For a virtual knot diagram $D$, we assign an orientation to the virtual knot.
We say that a classical crossing of $D$ is {\it positive} if
the understrand passes though the crossing from the right to the left with respect to the overstrand.
Otherwise, it is said to be {\it negative}. The sign $\sign(c)$ at a crossing point $c$ is defined to be $+1$
if it is positive and $-1$ otherwise.
A crossing change changes a positive crossing into a negative one and vice-versa.

\subsection{Generalized Reidemeister moves}

The local moves of virtual knot diagrams described in Figure~\ref{figRM} are called {\it generalized Reidemeister moves}.
The moves R1, R2, R3 are called classical Reidemeister moves. It is known that two classical knot diagrams represent 
the same knot if and only if these diagrams are related by classical Reidemeister moves.
For virtual knots, we have four additional moves VR1, VR2, VR3 and MR. The moves VR1, VR2, VR3 are called
{\it virtual Reidemeister moves} and the move MR is called a {\it mixed Reidemeister move}.
By definition, two virtual knot diagrams represent the same virtual knot if and only if these diagrams are
related by generalized Reidemeister moves.

\begin{figure}[htbp]
\begin{center}
  \includegraphics[width=14cm]{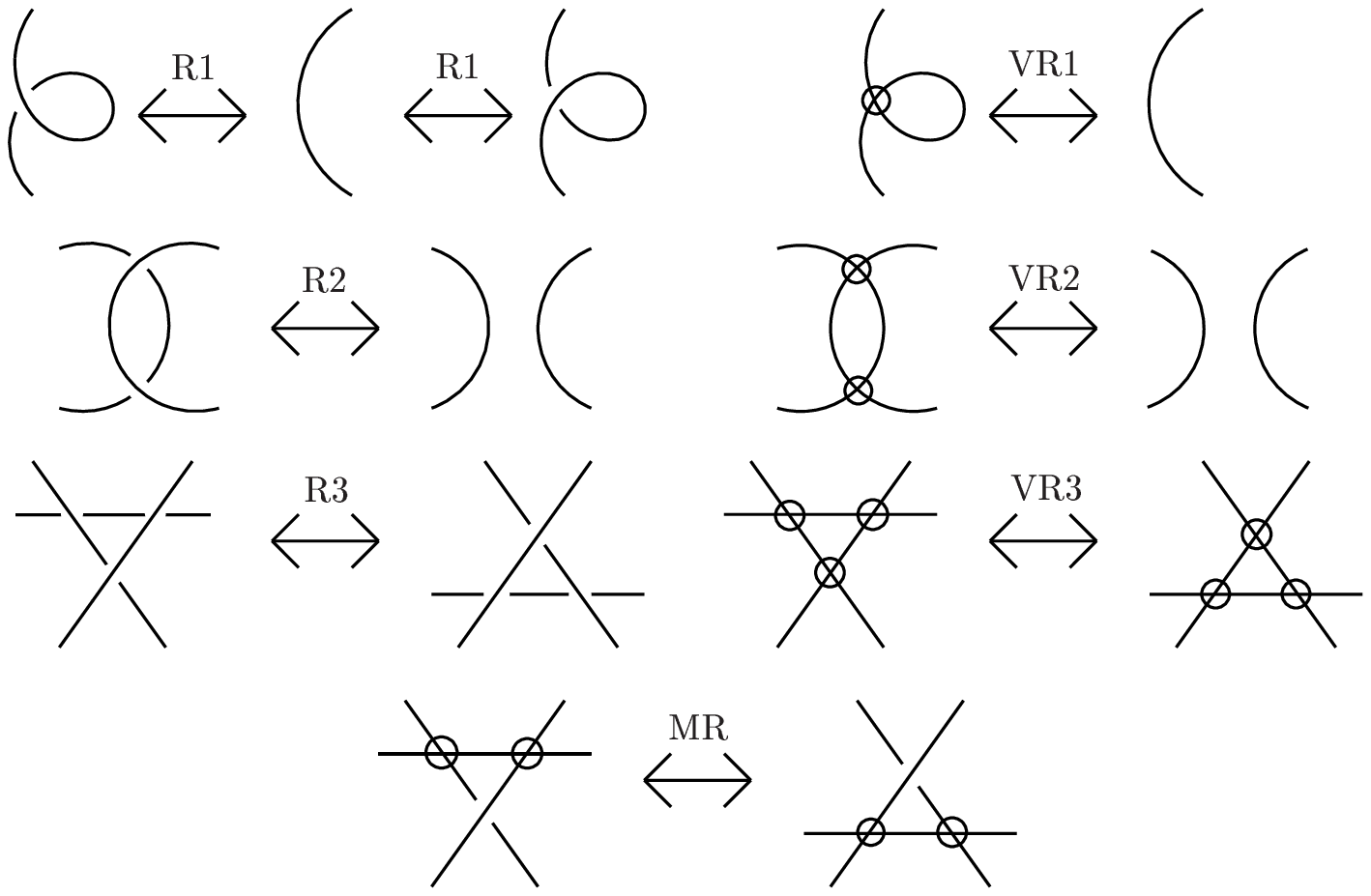}
  \caption{Generalized Reidemeister moves.\label{figRM}}
\end{center}
\end{figure}

Two virtual knots are said to be {\it virtually homotopic} if they are related by isotopy, generalized Reidemeister moves
and crossing changes. If a virtual knot is virtually homotopic to the unknot then we say that it is {\it virtually null-homotopic}.

\section{Upper bound}

\subsection{Virtual knot $(i,j,k)$}

Let $i,j$ and $k$ be integers such that $0<j$ and $0\leq k<i$.
We denote by $VB_{i,j}^1$ the virtual braid diagram obtained from the braid part of the standard diagram of 
the $(i,j)$-torus link by changing all classical crossings on the left-most overstrand into virtual crossings.
For example, the braid in Figure~\ref{fig1} is $VB_{7,3}^1$.
Let $B_k$ be the braid given by $\sigma_k\sigma_{k-1}\cdots\sigma_2\sigma_1$
as shown in Figure~\ref{fig100}, where we read the braid words from the left to the right. 
We denote by $(i,j,k)$ the virtual link diagram obtained as the closure of the product of $VB_{i,j}^1$ and $B_k$.
Hereafter, we always assume that $(i,j,k)$ is a virtual knot, that is, it has only one link component.
\begin{figure}[htbp]
\begin{center}
  \includegraphics[width=8cm]{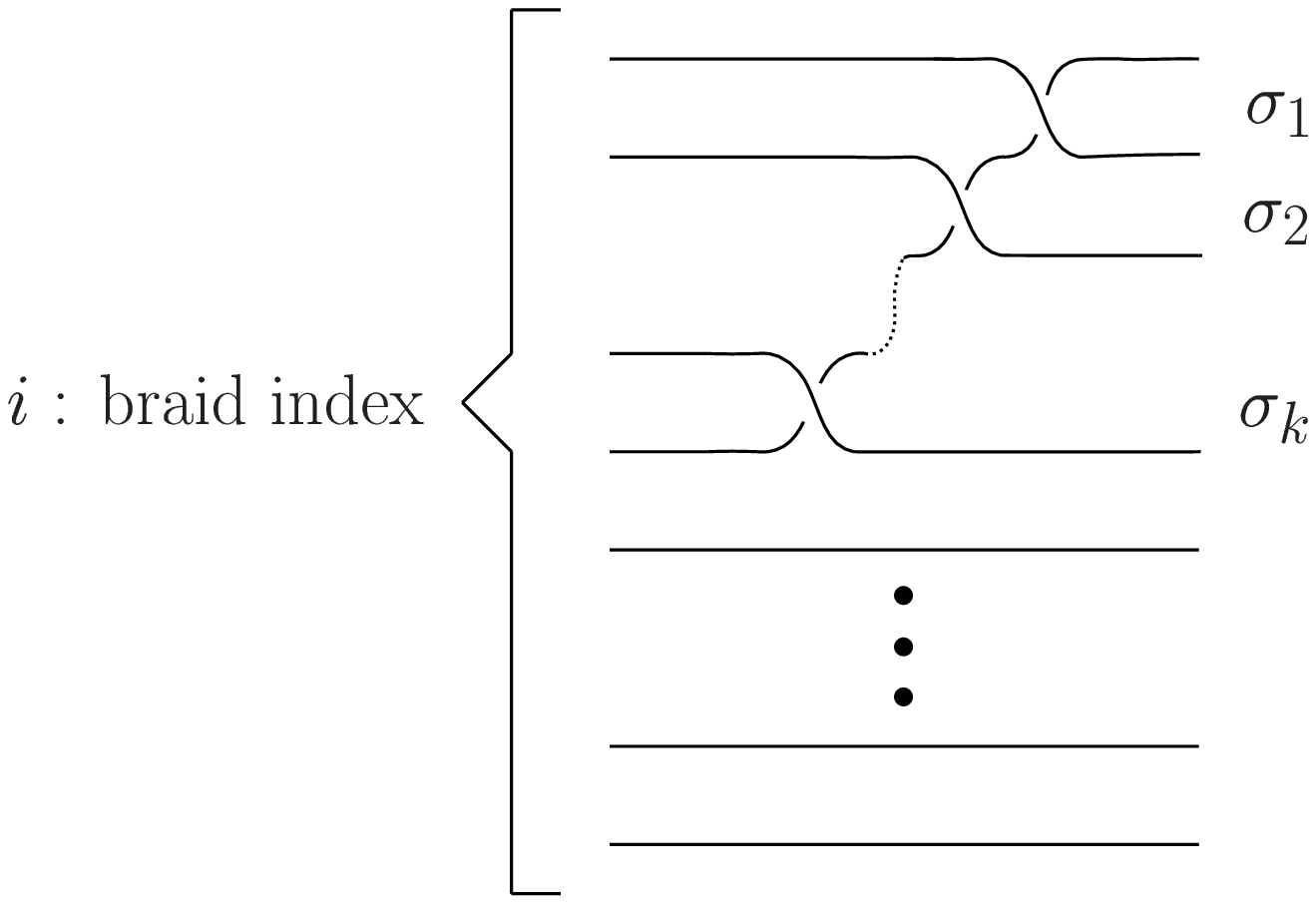}
  \caption{The braid $B_k$.\label{fig100}}
\end{center}
\end{figure}

\begin{rem}\label{rem99}
It is known by Masuda in~\cite[Lemma~5.8]{masuda} that if the supporting genus of a virtual knot is at most one 
then it is virtually null-homotopic. The proof is done by realizing the virtual knot as a knot on a torus
and straightening it in the torus.
If it is homotopic to a straight line in the torus then the corresponding virtual knot is virtually homotopic to the unknot.
If it is homotopic to a straight line with multiplicity $m>1$ then the corresponding virtual knot is virtually homotopic to
the $(m,1)$-torus knot, which is also the unknot. Therefore, the virtual knot is virtually null-homotopic.
We can easily check that the supporting genera of $VT_{p,q}^n$ and $(i,j,k)$ are at most one.
Therefore, they are virtually null-homotopic by his result.
\end{rem}

We will prove the following theorem:

\begin{thm}\label{thm2}
%The generalized unknotting number of the virtual knot $(i,j,k)$ is
%\[
%   U((i,j,k))=\left(0,\frac{(i-1)(j-1)+k}{2}\right).
%\]
%In other words, 
The virtual unknotting number of the virtual knot $(i,j,k)$ is 
\[
   \vu((i,j,k))=\frac{(i-1)(j-1)+k}{2}.
\]
\end{thm}

Setting $k=0$, we have Theorem~\ref{thm1}.

The aim of this section is to prove the upper bound of the equality in Theorem~\ref{thm2}, that is, 
\begin{equation}\label{eq1}
  \vu((i,j,k))\leq \frac{(i-1)(j-1)+k}{2}.
\end{equation}

\subsection{Reduce to the case $j\leq i$}

\begin{lemma}\label{lemma1}
Suppose that $j>i$.
\begin{itemize}
\item[(1)] If $(i,j-i,k)$ consists of one link component then $(i,j,k)$ also does.
\item[(2)] $(i,j,k)$ is modified into $(i,j-i,k)$ by $i(i-1)/2$ crossing changes.
\end{itemize}
\end{lemma}

\begin{proof}
The braid part of the virtual knot $(i,j,k)$ is the product of $VB_{i,j-i}^1$, the full twist $B_{i,i}$ of
$i$ strands and the braid $B_k$.
We can replace the full twist $B_{i,i}$ into $i$ horizontal strands by $i(i-1)/2$ crossing changes.
Hence the assertions hold.
\end{proof}

\begin{lemma}\label{lemma2}
If $(i,j-i,k)$ satisfies inequality~\eqref{eq1} then $(i,j,k)$ satisfies it also.
\end{lemma}

\begin{proof}
The inequality~\eqref{eq1} for $(i,j-i,k)$ is $\vu((i,j-i,k))\leq ((i-1)(j-i-1)+k)/2$.
Hence, by Lemma~\ref{lemma1}, we have
\[
\begin{split}
   \vu((i,j,k))&\leq \vu((i,j-i,k))+\frac{i(i-1)}{2} \\
   &\leq \frac{(i-1)(j-i-1)+k}{2}+\frac{i(i-1)}{2}=\frac{(i-1)(j-1)+k}{2}.
\end{split}
\]
\end{proof}

Thus, to prove the upper bound~\eqref{eq1},
it is enough to show the inequality in the case where $j\leq i$.

\subsection{Operations A, B and C}

%In the following discussion, we will assume that $(i,j,k)$ is a virtual knot.
%If $k=0$ then $i$ and $j$ are necessarily coprime, though we do not assume this if $k>0$.

We will give an unknotting sequence of $(i,j,k)$ explicitly.
To do this, we need to introduce three operations, named A, B and C.
First we introduce operations 1, 2 and 3 to define operations A and B, 
and then give the definitions of operations A, B and C.
%In the following figures, we always close the braid by describing strands for closure such that
%they pass though the top of the braid diagram (though it is not important since the diagrams are described on $S^2$).

\begin{dfn}[Operation 1]
Let $(VB_{i,j}^1)B_k$ be the braid part of $(i,j,k)$. Starting from the right-top point, we follow the strand to the left
until we meet a virtual crossing. 
Assume that this strand is always an understrand at each crossing.
We move this strand to the top of the braid part as shown in Figure~\ref{figM1}.
This operation is called {\it operation 1}. We only use classical Reidemeister moves during this operation.
\end{dfn}

\begin{figure}[htbp]
\begin{center}
  \includegraphics[height=40mm]{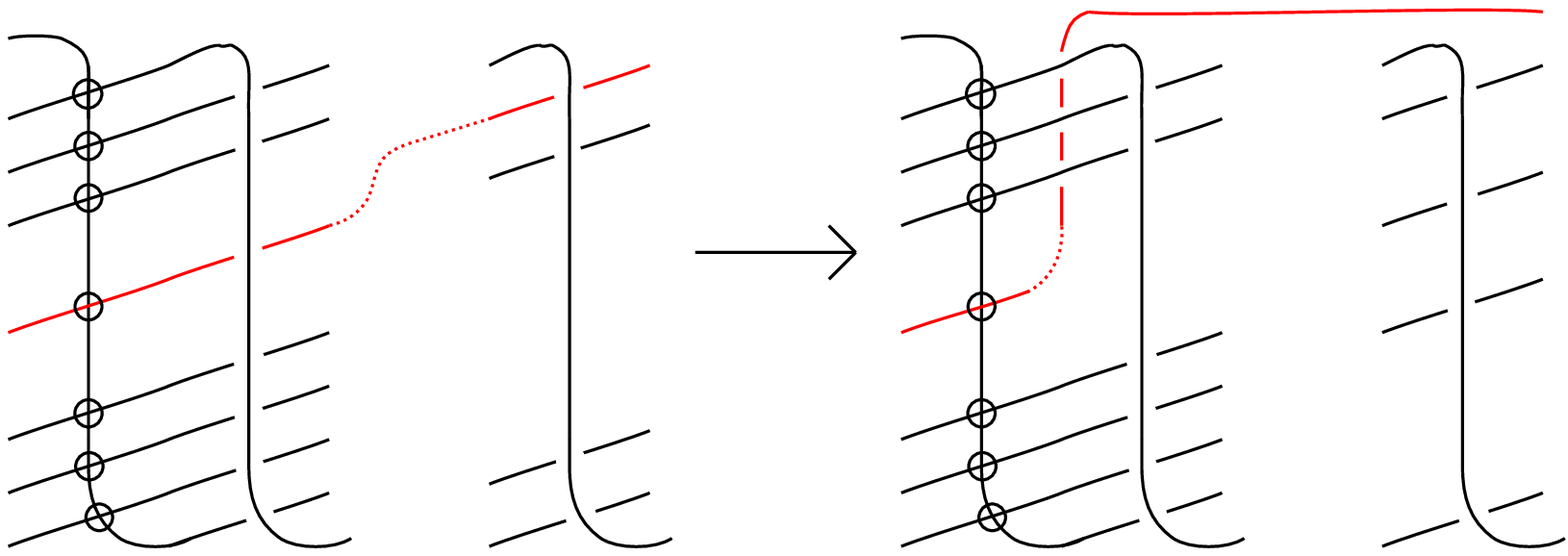}
  \caption{Operation 1.\label{figM1}}
\end{center}
\end{figure}

\begin{dfn}[Operation 2]
Let $(VB_{i,j}^1)B_k$ be the braid part of $(i,j,k)$. Starting from the right-top point, we follow the strand to the left
until we meet a virtual crossing. Assume that this strand contains some of the overstrands of $VB_{i,j}^1$ other
than the most-right overstrand of $VB_{i,j}^1$.
Remark that, since we are assuming $j\leq i$, this strand contains only one overstrand
(passing through $i-1$ classical crossings).
We apply crossing changes to all the crossings on this overstrand and 
then move the strand to the top of the braid part as shown in Figure~\ref{figM2}.
This operation is called {\it operation 2}. 
We use $i-1$ crossing changes during this move.
\end{dfn}

\begin{figure}[htbp]
\begin{center}
  \includegraphics[height=35mm]{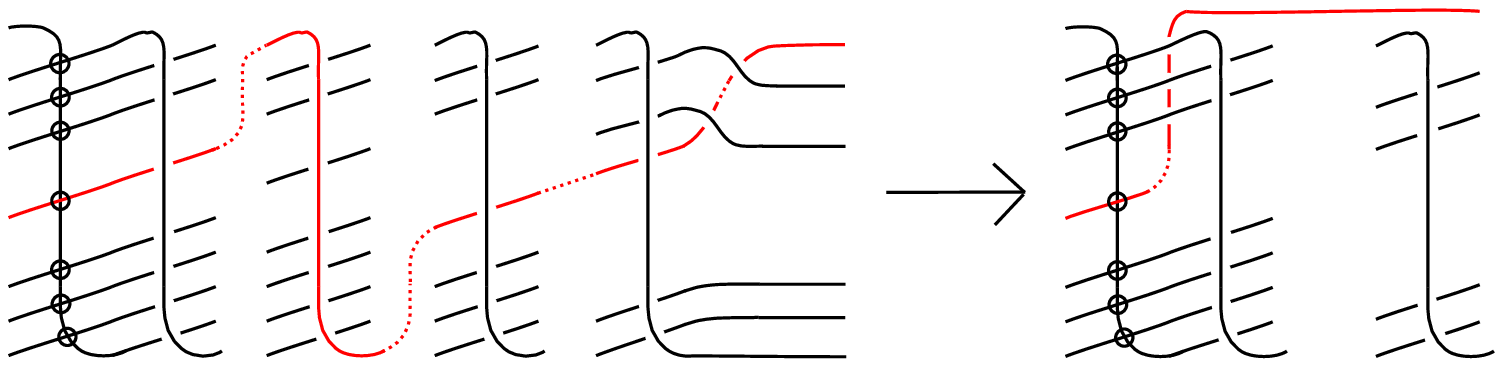}
  \caption{Operation 2. We use crossing changes $i-1$ times.\label{figM2}}
\end{center}
\end{figure}

\begin{dfn}[Operation 3]
Take a diagram obtained by operation 1 or operation 2. We apply mixed Reidemeister moves, 
apply VR1 move once, and then
move all the crossing lying on the left of the virtual crossings to the most-right position by the conjugation of the braid.
We then obtain a virtual knot diagram of the form $(i,j,k)$ again.
This operation is called {\it operation 3}. See Figure~\ref{figM3}. We only use generalized 
Reidemeister moves during this operation.
The braid index of the obtained virtual braid diagram is $i-1$ since we use VR1 move once.
\end{dfn}

\begin{figure}[htbp]
\begin{center}
  \includegraphics[height=8cm]{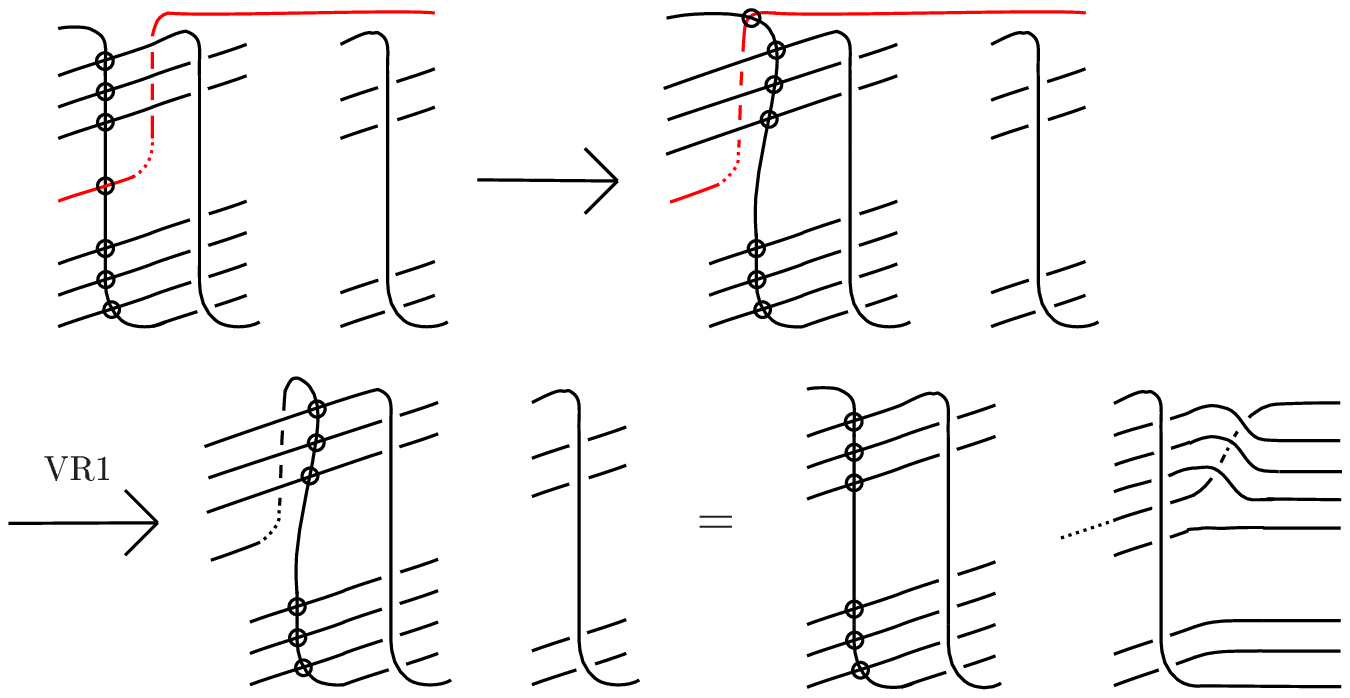}
  \caption{Operation 3. We use VR1 move once, so that the braid index decreases.\label{figM3}}
\end{center}
\end{figure}

\begin{dfn}[Operation A]
Assume that $k+j<i$.
We first apply operation 1 and then operation 3, which modifies $(i,j,k)$ into $(i-1,j,j+k-1)$.
This is called {\it operation~A}. 
Note that we only use generalized Reidemeister moves during this operation.
In other words, we do not use crossing changes during operation A.
\end{dfn}

\begin{dfn}[Operation B]
Assume that $i<k+j<2i-1$ and $i\ne k+1$.
We first apply operation 2 and then operation 3, which modifies $(i,j,k)$ into $(i-1,j-1,j+k-i-1)$.
This is called {\it operation B}. We use $i-1$ crossing changes during this operation.
\end{dfn}

\begin{dfn}[Operation C]
Assume that $i=k+1$ and $j\geq 2$. 
Operation C is the move in Figure~\ref{figMC} which modifies $(i,j,k)$ into $(i,j-1,0)$.
We use $i-1$ crossing changes during this operation.
\end{dfn}

\begin{figure}[htbp]
\begin{center}
  \includegraphics[height=7cm]{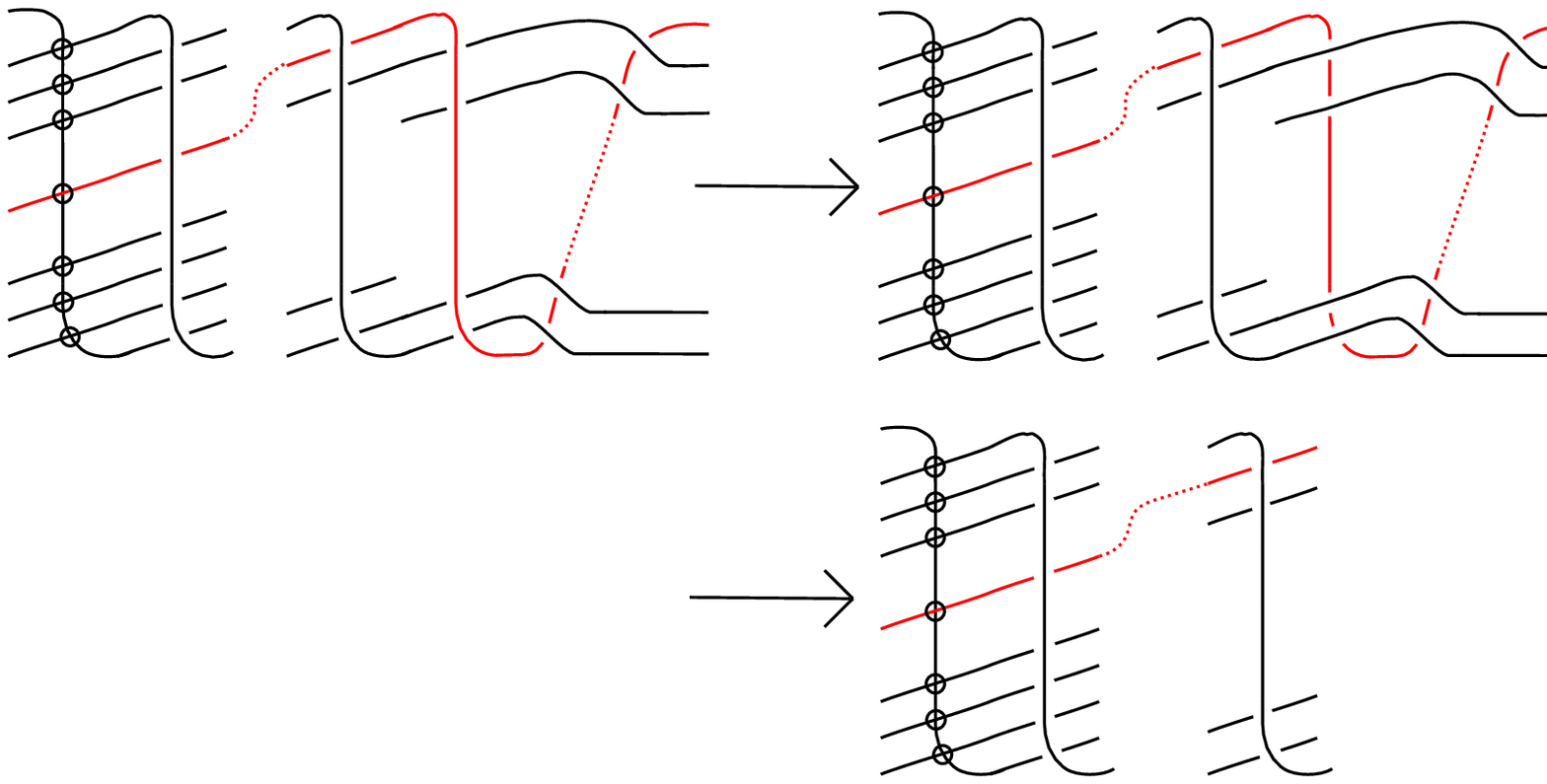}
  \caption{Operation C. We use crossing changes $i-1$ times.\label{figMC}}
\end{center}
\end{figure}

\subsection{An unknotting sequence of $(i,j,k)$}

\begin{lemma}\label{lemma10}
Assume that $(i,j,k)$ is a virtual knot with $2\leq j\leq i$ and $0\leq k<i$.
Then we can apply one of operations A, B and C.
More precisely, 
\begin{itemize}
\item[(i)] if $2\leq k+j<i$ then we apply operation A,
\item[(ii)] if $i<k+j< 2i-1$ and $i\ne k+1$ then we apply operation B, 
\item[(iii)] if $i<k+j< 2i-1$ and $i=k+1$ then we apply operation C, and 
\item[(iv)] if $i=j=k+1$ then we apply operation C.
\end{itemize}

\end{lemma}

\begin{proof}
If $i=j$ then $k=i-1$, otherwise $(i,j,k)$ has more than one link component.
This is in case (iv) and we can apply operation C.

Suppose that $2\leq j\leq i-1$ and $0\leq k < i$. We decompose it into four cases:
(i), (ii), (iii) in the assertion and the case where $i=k+j$.
If it is in case (i), (ii) and (iii) then we apply operation A, B and C, respectively.
Suppose $i=k+j$. Let $VB_{i,j}^1$ is the braid part of $VT_{i,j}^1$, i.e.,
$VT_{i,j}^1$ is the closure of $VB_{i,j}^1$.
If we follow the strand of $VB_{i,j}^1$ from the left-top then on the right side of the braid $VB_{i,j}^1$,
we arrive at the $j-i+1$-th strand counted from the top. The braid part of $(i,j,k)$ is the product of $VB_{i,j}^1$
and $B_k$, and the $j-i+1$-th strand is connected to the first strand on the right side of $B_k$.
Taking the closure of this braid, we see that this constitutes a link component of $(i,j,k)$.
Since $i\geq 2$, it has more than one link component. Thus we can exclude this case.
\end{proof}

\begin{lemma}\label{prop1}
Suppose that $(i,j,k)$ is a virtual knot with $2\leq j\leq i$ and $0\leq k<i$.
Then $(i,j,k)$ can be modified to the unknot by a sequence of operations A, B and C.
\end{lemma}

\begin{proof}
First we observe the case where $i=2$. Since $(i,j,k)$ has only one link component, possible cases are
only $(2,2,1)$ and $(2,1,0)$. The virtual knot $(2,1,0)$ is the unknot.
The virtual knot $(2,2,1)$ is modified into $(2,1,0)$ by operation C. Thus it is virtually null-homotopic.

Assume that $i\geq 3$.
We check that the conditions $1\leq j\leq i$ and $0\leq k<i$ are satisfied after operations A, B and C.
If it is in case (i) in Lemma~\ref{lemma10} then we apply operation A and obtain $(i-1, j, j+k-1)$.
The second inequality is verified as $(i-1)-(j+k-1)=i-j-k>0$.
If it is in case (ii) in Lemma~\ref{lemma10} then we apply operation B and obtain $(i-1, j-1, j+k-i-1)$.
The second inequality is verified as $(i-1)-(j+k-i-1)=2i-j-k>0$.
If it is in case (iii) or (iv) in Lemma~\ref{lemma10} then we apply operation C and obtain $(i, j-1, 0)$,
which satisfies the required inequalities. 

By these operations, either the first or the second entry decreases one by one. 
Since $2\leq j\leq i$, we will reach either case (a)~$i=3$ or case (b)~$i>3$ and $j=1$.
If it is in case~(a) then either $(i,j,k)=(3,2,2)$, $(3,2,0)$ or $(3,1,0)$.
For $(3,2,2)$, we obtain $(2,1,0)$ by operation B. For $(3,2,0)$, we obtain $(2,2,1)$ by operation A
and then $(2,1,0)$ by operation C. Since the virtual knot $(i,1,0)$ is the unknot, the knot is virtually null-homotopic in either case 
in case~(a).
If it is in case~(b) and $k>0$ then we can check directly that it always has more than one link component, 
and we can exclude this case.
If it is in case~(b) and $k=0$ then it is the unknot $(i,1,0)$.
Thus we have the assertion.
\end{proof}

\subsection{Proof of the upper bound}

In the following proposition, we do not assume $j\leq i$.

\begin{prop}\label{prop2}
Suppose that $(i,j,k)$ is a virtual knot with $0\leq k<i$. Then inequality~\eqref{eq1} holds.
%\begin{equation}\label{eq2}
%   \vu((i,j,k))\leq \frac{(i-1)(j-1)+k}{2}.
%\end{equation}
\end{prop}

\begin{proof}
By Lemma~\ref{lemma2}, it is enough to show the assertion in the case where $j\leq i$.

As we saw in the proof of Proposition~\ref{prop1}, we can reach the unknot of the form $(i,1,0)$ after applying 
operations A, B and C. For $(i,1,0)$, the right hand side of inequality~\eqref{eq1} is $0$. Thus the assertion holds.

Now we prove the assertion by induction. Let $m(i,j,k)$ be the number of operations A, B and C to obtain the unknot.
For the induction, we assume that inequality~\eqref{eq1} holds for any virtual knot $(i,j,k)$ with $m(i,j,k)\leq \ell$,
and check the inequality for a virtual knot $(i',j',k')$ with $m(i',j',k')=\ell+1$. 

Suppose that $(i,j,k)$ is obtained from $(i',j',k')$ by operation A, i.e., in the case $(i',j',k')=(i+1,j,k-j+1)$.
We do not use crossing changes during operation A. Since 
\[
   \frac{(i'-1)(j'-1)+k'}{2}=\frac{(i+1-1)(j-1)+k-j+1}{2}=\frac{(i-1)(j-1)+k}{2},
\]
the assertion holds.

Suppose that $(i,j,k)$ is obtained from $(i',j',k')$ by operation B, i.e., in the case $(i',j',k')=(i+1,j+1,k-j+i+1)$.
We use $i'-1$ crossing changes during operation B. Since 
\[
   \frac{(i'-1)(j'-1)+k'}{2}-(i'-1)=\frac{ij+k-j+i+1}{2}-i=\frac{(i-1)(j-1)+k}{2},
\]
the assertion holds.

Finally, suppose that $(i,j,0)$ is obtained from $(i',j',k')$ by operation C, i.e., in the case $(i',j',k')=(i,j+1,i-1)$.
We use $i'-1$ crossing changes during operation B. Since 
\[
   \frac{(i'-1)(j'-1)+k'}{2}-(i'-1)=\frac{(i-1)j+i-1}{2}-(i-1)=\frac{(i-1)(j-1)}{2},
\]
the assertion holds. This completes the proof.
\end{proof}

\section{Lower bound and Proof of Theorem~\ref{thm2}}

\subsection{Gauss diagram}

We shortly introduce the Gauss diagram of a virtual knot.
Let $D$ be a virtual knot diagram. Choose a starting point $p$ on $D$. Describe the unit circle $\Gamma$ on a plane
and choose a point $q$ on $\Gamma$ which corresponds to $p$. We follow the knot strand of $D$ from $p$ and
also $\Gamma$ from $q$ simultaneously.
When we meet a classical crossing on $D$ then we put a label on $\Gamma$.
In consequence, we have $2N$ points on $\Gamma$, where $N$ is the number of classical crossings on $D$.
Each classical crossing of $D$ corresponds to two points among these $2N$ points.
We connect these two points by an edge and add an arrowhead to the point on $\Gamma$ corresponding to the understrand.
Finally, we write a sign $+$ if the crossing is positive and $-$ otherwise.
The obtained diagram is called the {\it Gauss diagram} of $D$ and a signed, arrowed edge on the Gauss diagram 
is called a {\it chord}.

A {\it flip} is an operation on a Gauss diagram which reverses the direction of the arrow of 
a chord and changes the sign assigned to the chord.
This operation corresponds to a crossing change on a virtual knot diagram obtained from the Gauss diagram.

\subsection{Invariant $P(K)$}

Let $GD(K)$ be a Gauss diagram of a virtual knot $K$.
The endpoints of a chord $c$ of $GD(K)$ decompose the circle $\Gamma$ of 
$GD(K)$ into two open arcs, $\gamma_1$ and $\gamma_2$.
We flip suitable chords in $GD(K)$ other than $c$ such that their arrowheads are on $\gamma_1$
and denote the sum of signs in the flipped Gauss diagram by $i(c)$.
We then define $P(GD(K))$ for the Gauss diagram $GD(K)$ by
\[
   P(GD(K))=\sum_{c}\sign(c)t^{|i(c)|},
\]
where the sum runs all chords $c$ such that $i(c)\ne 0$.

\begin{thm}[see Definition~3.1 in~\cite{chernov}]
The polynomial $P(GD(K))$ does not depend on a choice of Gauss diagram of the virtual knot $K$.
\end{thm}

\begin{dfn}
The $P$-invariant $P(K)$ of a virtual knot $K$ is defined 
by $P(K)=P(GD(K))$ for a Gauss diagram $GD(K)$ of $K$.
\end{dfn}

Note that this invariant was first introduced by Henrich in a slightly different form, see~\cite[Definition~3.3]{henrich}.

\begin{lemma}[Byberi-Chernov~\cite{chernov}]\label{lemma13}
Set $P(K)=\sum_{m\in\N}b_mt^m$. Suppose that $K$ is virtually null-homotopic.
Then the following inequality holds:
\[
   \vu(K)\geq\frac{1}{2}\sum_{m\in\N}|b_m|.
\]
\end{lemma}

\subsection{Proofs of the lower bound and Theorem~\ref{thm2}}

In this section, we complete the proof of Theorem~\ref{thm2} by showing
the lower bound of the equality in the theorem.

The definition of $i(c)$ is interpreted in terms of a virtual knot diagram $D$ as follows:
A chord $c$ corresponds to a classical crossing $p$ of $D$. Start from $p$ and follow the strand into one of 
the possible directions, then we come back to $p$ again at some moment. Denote this loop by $\gamma_1$
and the rest part of the virtual knot by $\gamma_2$.
For each classical crossing consisting of $\gamma_1$ and $\gamma_2$, if $\gamma_1$ is an understrand then
we apply a crossing change so that $\gamma_1$ is an overstrand.
We denote the obtained virtual knot diagram by $D_c$.
Then, $i(c)$ is the sum of signs of all classical crossings consisting of $\gamma_1$ and $\gamma_2$.
Let $i_o(\gamma_1,\gamma_2)$ and $i_u(\gamma_1,\gamma_2)$ be the number of classical crossings on $D$
consisting of $\gamma_1$ and $\gamma_2$ such that $\gamma_1$ is an overstand and an understrand, respectively.
Then we have
\[
|i(c)|=|i_o(\gamma_1,\gamma_2)-i_u(\gamma_1,\gamma_2)|.
\]

A closed positive braid is a link, or a diagram, obtained as the closure of a braid consisting of only positive words.
Remark that all classical crossings on $(i,j,k)$ are positive and they are given by positive words in
its braid word presentation.

\begin{lemma}\label{lemma11}
Let $B$ be the diagram of a closed positive braid and suppose that $B$ is a knot.
Let $GD(B)$ be the Gauss diagram of $B$. Then $i(c)=0$ for any chord $c$ in $GD(B)$.
\end{lemma}

\begin{proof}
Choose a crossing point $p$ of $B$. We choose $\gamma_1$ by following the strand from $p$ in the right-bottom direction.
There are two important remarks: 
One is that, since all crossings on $B$ are positive, $\gamma_1$ is always an overstrand when $\gamma$ directs to the right-bottom,
and it is always an understrand when $\gamma$ directs to the right-top.
The other remark is that, since $B$ is a knot, $\gamma_1$ comes back to $p$ from the left-bottom.
If $\gamma_1$ has no self-crossing, then $i_o(\gamma_1,\gamma_2)=i_u(\gamma_1,\gamma_2)$.
Furthermore, this equality holds even if $\gamma_1$ has self-crossings since each self-crossing 
consists of a strand going to the right-top and a strand going to the right-bottom. Thus we have $i(c)=0$.
\end{proof}

\begin{lemma}\label{lemma12}
Suppose that $(i,j,k)$ is a virtual knot diagram and let $GD((i,j,k))$ be its Gauss diagram.
Then $i(c)\ne 0$ for any chord $c$ in $GD((i,j,k))$.
\end{lemma}

\begin{proof}
Assume that there exists a chord $c$ such that $i(c)=0$. Choose $\gamma_1$ such that it does not contain the vertical strand 
$a_1$ which passes though all virtual crossings. This is possible by exchanging the roles of $\gamma_1$ and $\gamma_2$ 
if necessary.
Suppose that $\gamma_1$ meets $a_1$ at virtual crossings $N$ times, where $N$ is at least $1$.
If we regard the virtual crossings as classical crossings, then the diagram is a closed positive braid
and we have $i(c)=0$ by Lemma~\ref{lemma11}. 
This means that $i(c)$ for $(i,j,k)$ is calculated from that for the closed positive braid
by subtracting the signed number of virtual crossings on $\gamma_1$. Thus we have $|i(c)|=N\geq 1$.
\end{proof}

\begin{proof}[Proof of Theorem~\ref{thm2}]
As mentioned in Remark~\ref{rem99}, $(i,j,k)$ is virtually null-homotopic.
This also follows from Lemma~\ref{prop1}.
The upper bound of the equality in Theorem~\ref{thm2} follows from Proposition~\ref{prop2}.
For the lower bound, we use the inequality in Lemma~\ref{lemma13}.
The diagram $(i,j,k)$ has $(i-1)(j-1)+k$ classical crossings and the signs of all classical crossings are positive.
Since $i(c)\ne 0$ for any chord $c$ by Lemma~\ref{lemma12}, we have $\sum_{m\in\N}|b_m|=(i-1)(j-1)+k$.
Thus the lower bound follows from the inequality in Lemma~\ref{lemma13}.
\end{proof}

\begin{rem}
The assertion in Theorem~\ref{thm1} holds in case $pq<0$ also
since the virtual knot diagram $V_{p,q}^1$ with $pq<0$ is the mirror image of $V_{|p|,|q|}^1$.
\end{rem}

\section{Further discussions}

\subsection{The virtual torus knot $VT_{p,q}^n$}\label{sec51}

\begin{lemma}\label{lemma14}
Suppose that $p$ and $q$ are coprime and $n$ is an integer such that $1\leq n\leq q$.
If $p$ and $n$ are coprime then $\vu(VT_{p,q}^n)\geq (p-1)(q-n)/2$.
\end{lemma}

\begin{proof}
Let $V_{p,n}$ be the braid part of $VT_{p,q}^n$ having only virtual crossings.
If $n\geq p$ then we can apply generalized Reidemeister moves to $V_{p,n}$ such that it becomes $V_{p,n-p}$.
Therefore, we can assume that $n<p$.

Assume that there exists a chord $c$ such that $i(c)=0$. Choose $\gamma_1$ such that it does not contain the vertical strand 
$a_1$. This is possible by exchanging the roles of $\gamma_1$ and $\gamma_2$ if necessary.
Observing how $\gamma_1$ passes though $V_{p,n}$, we have two cases:
\begin{itemize}
\item[(1)] When $\gamma_1$ passes a vertical strand $a_i$, for some $i\in\{1,\ldots,n\}$, the contribution 
on the virtual crossings on $a_i$ to $i_o(\gamma_1,\gamma_2)$ by virtualizations 
from the closed positive braid $D_{p,q}$ is $p-1$.
The strand containing $a_i$ in the braid part $V_{p,n}$ also has the contribution $n-1$ to $i_u(\gamma_1,\gamma_2)$
since it has one crossing with each of the other vertical strands $a_j$, $j\ne i$, in $V_{p,n}$.
\item[(2)] When $\gamma_1$ passes though $V_{p,n}$ horizontally, the contribution to $i_u(\gamma_1,\gamma_2)$ is $n$
and there is no contribution to $i_o(\gamma_1,\gamma_2)$.
\end{itemize}
Let $N_1$ and $N_2$ be the number of case (1) and (2) appearing along $\gamma_1$, respectively.
Since $\gamma_1$ does not contain $a_1$, $0\leq N_1\leq n-1$. We have $1\leq N_2\leq p-n$.
As in the proof of Lemma~\ref{lemma12}, we check the difference of $i(c)$ for $VT_{p,q}^n$ and $D(T_{p,q})$,
where $i(c)=0$ for $D(T_{p,q})$ by Lemma~\ref{lemma11}.
From the above observation, we have 
\[
\begin{split}
   0&=|i(c)|=|i_o(\gamma_1,\gamma_2)-i_u(\gamma_1,\gamma_2)| \\
   &=|(p-1)N_1-((n-1)N_1+nN_2)|=|pN_1-n(N_1+N_2)|.
\end{split}
\]
Thus we have $pN_1=n(N_1+N_2)$. We then divide the both sides by $n$ and get $pN_1\equiv 0\mod n$.
Since $p$ and $n$ are coprime and $0\leq N_1\leq n-1$, we have $N_1=N_2=0$. This is a contradiction. 

Since $i(c)\ne 0$ for any chord $c$, the assertion holds by Lemma~\ref{lemma13} as in the proof of Theorem~\ref{thm2}.
\end{proof}

We checked the lower bounds of $\vu(VT_{p,q}^2)$ obtained by the $P$-invariant and
also tried to find good upper bounds by moves of virtual knot diagrams for virtual knots $VT_{p,q}^2$ with $q<p\leq 8$.
The result is shown in Table~\ref{table1}.

\begin{table}[htbp]
\begin{center}
  \begin{tabular}{|c|c|c|c|}
\hline
$(p,q)$ & $\sum|b_j|/2$ & upper bound & $\vu(VT_{p,q}^2)$ \\
\hline\hline
$(3,2)$ & $0$ & $0$ & $\vu(VT_{3,2}^2)=0$ \\
\hline
$(4,3)$ & $1$ & $1$ & $\vu(VT_{4,3}^2)=1$ \\
\hline
$(5,2)$ & $0$ & $0$ & $\vu(VT_{5,2}^2)=0$ \\
\hline
$(5,3)$ & $2$ & $2$ & $\vu(VT_{5,3}^2)=2$ \\
\hline
$(5,4)$ & $4$ & $4$ & $\vu(VT_{5,4}^2)=4$ \\
\hline
$(6,5)$ & $7$ & $7$ & $\vu(VT_{6,5}^2)=7$ \\
\hline
$(7,2)$ & $0$ & $0$ & $\vu(VT_{7,2}^2)=0$ \\
\hline
$(7,3)$ & $3$ & $3$ & $\vu(VT_{7,3}^2)=3$ \\
\hline
$(7,4)$ & $6$ & $6$ & $\vu(VT_{7,4}^2)=6$ \\
\hline
$(7,5)$ & $9$ & $9$ & $\vu(VT_{7,5}^2)=9$ \\
\hline
$(7,6)$ & $12$ & $12$ & $\vu(VT_{7,6}^2)=12$ \\
\hline
$(8,3)$ & $3$ & $3$ & $\vu(VT_{8,3}^2)=3$ \\
\hline
$(8,5)$ & $9$ & $10$ & $9 \leq \vu(VT_{8,5}^2)\leq 10$ \\
\hline
$(8,7)$ & $17$ & $17$ & $\vu(VT_{8,7}^2)=17$ \\
\hline
  \end{tabular}
 \caption{Virtual unknotting numbers of $VT_{p,q}^2$}\label{table1}
  \end{center}
\end{table}

If $p$ is odd then the lower bound is given in Lemma~\ref{lemma14}.
If $p$ is even then we calculate the $P$-invariant for each case.
The upper bound is obtained by finding moves of virtual knot diagrams,
which are written in~\cite{yanagi}.
The virtual unknotting numbers of $V_{p,q}^3$ with $q<p\leq 8$ are also discussed in~\cite{yanagi}.

\subsection{A virtual torus knot which is not virtually null-homotopic}\label{sec52}

As mentioned in Remark~\ref{rem99}, the virtual torus knot $VT_{p,q}^n$ is virtually null-homotopic.
On the other hand, we can check that the virtual torus knot in Figure~\ref{fig22} is not virtually homotopic to the unknot by the following reason:
For a Gauss diagram $GD(K)$ of a virtual knot $K$, let $\overline{GD}(K)$ be the Gauss diagram obtained from $GD(K)$ 
by flipping some of chords such that all signs are $+1$. 
If $K$ is a virtual knot obtained from a closed positive braid by virtualizations
then $GD(K)=\overline{GD}(K)$ for the Gauss diagram obtained from that diagram of $K$.
In particular, $VT_{p,q}^n$ has this property.
For each chord $c$, we rotate the Gauss diagram $\bar D$ such that $c$ directs to the top, and 
let $n_+(c)$ be the number of chords which cross the chord $c$ from the left to the right 
and $n_-(c)$ be the number of chords which cross $c$ from the right to the left. Set $n(c)=n_+(c)-n_-(c)$.
It is known in~\cite{chernov}, originally from~\cite{turaev}, that the polynomial
\[
   u(K)=\sum_c\sign(n(c))t^{|n(c)|}
\]
does not depend on the choice of a Gauss diagram of $K$ and also does not change under crossing changes,
where the sum runs all chords $c$ on $\bar D$ with $n(c)\ne 0$. This polynomial is called the {\it $u$-invariant} of $K$.
In, particular, if a virtual knot $K$ is virtually null-homotopic then $u(K)=0$.
We can check that the $u$-invariant of the virtual torus knot in Figure~\ref{fig22} is $t^2-2t$ and hence it is not virtually null-homotopic.

\begin{figure}[htbp]
\begin{center}
  \includegraphics[width=5.5cm]{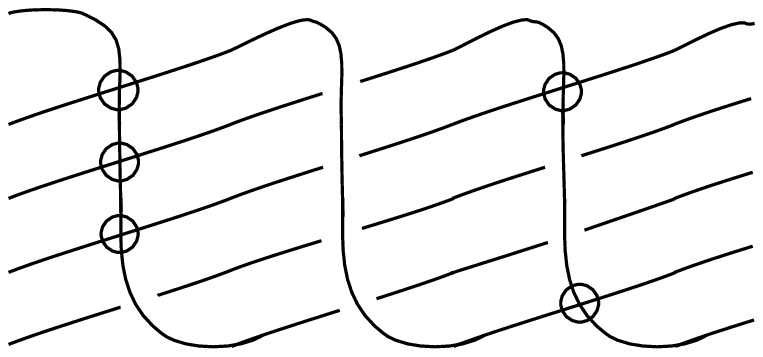}
  \caption{The closure of this braid is a virtual torus knot which is not virtually homotopic to the unknot.\label{fig22}}
\end{center}
\end{figure}

\subsection{On virtual Gordian complex}

For the set of virtually null-homotopic virtual knots, we can define a virtual version of the Gordian complex
as follows: set a vertex for each virtually null-homotopic virtual knot
and connect two vertices by an edge if one can be obtained from the other by one crossing change.
We then attach a $2$-simplex along a triangle consisting of three edges, and continue this process 
$m$-simplices with $m\geq 3$ inductively.
By construction, if there is an $m$-simplex then 
the $m+1$ virtual knots corresponding to the vertices of the $m$-simplex are related by a single crossing change.
The Gordian complex of classical knots is embedded in this complex canonically.
However, we do not know if the embedding preserves the distance of two vertices
(i.e., the minimal number of edges between two vertices). For example, we may ask the following:
Is the distance from the vertex corresponding to the $(p,q)$-torus knot to the vertex of the unknot is
$(p-1)(q-1)/2$ in the virtual Gordian complex?
If this is true then our main theorem means that 
the virtualization from $D(T_{p,q})$ to $VT_{p,q}^1$ does not change the distance
from the vertices of these knots to the vertex of the unknot in the virtual Gordian complex.

For another definition of a Gordian complex for virtual knots, see~\cite{ohyama1, ohyama2}. 
They used forbidden moves to define the Gordian complex. In particular, it is defined for all virtual knots.

\end{document}